\documentclass[reqno]{amsart}
\usepackage[psamsfonts]{amssymb}
\usepackage{amsthm}
\usepackage{color}

\newtheorem{Theorem}{Theorem}
\newtheorem{lemma}{Lemma}

\theoremstyle{remark}

\newtheorem{Definition}{Definition}

 \theoremstyle{Definition}
 \theoremstyle{remark}
 \newtheorem{ex}{Example}
 \numberwithin{equation}{section}

\usepackage{graphicx}%
\usepackage{multirow}%
\usepackage{amsmath,amssymb,amsfonts}%
\usepackage{amsthm}%
\usepackage{mathrsfs}%
\usepackage[title]{appendix}%
\usepackage{xcolor}%
\usepackage{textcomp}%
\usepackage{manyfoot}%
\usepackage{booktabs}%
\usepackage{algorithm}%
\usepackage{algorithmicx}%
\usepackage{algpseudocode}%
\usepackage{listings}%

 \usepackage{graphicx,amssymb,upref}
\usepackage{graphicx}
\usepackage{float}
\usepackage{enumerate}
\usepackage{hyperref}
\usepackage[utf8]{inputenc}
\usepackage[english]{babel}
\usepackage[square,numbers]{natbib}
\def\bege{\begin{equation}} \def\ende{\end{equation}}

   \def\begr{\begin{eqnarray}}
\def\endr{\end{eqnarray}} 
 
\def\bege{\begin{equation}} \def\ende{\end{equation}}
\def\begr{\begin{eqnarray}} \def\endr{\end{eqnarray}}
\def\bnum{\begin{enumerate}} \def\enum{\end{enumerate}}

\allowdisplaybreaks[4]

\begin{document}

\title[Complex symmetric operator ]{Commutants of complex symmetric weighted composition operators on Fock space}
\thanks{The author is thankful to  DST(India) for WISE-PDF  (File no. DST/WISE-PDF/PM-46/2023.\\ $^1$The corresponding author}

\author[A. Sharma]{$^1$Aakriti
 Sharma}

\address{
Department of Mathematics, University of Jammu-180006, J\&{K},
 India.}

\email{aakritishma321@gmail.com}

\begin{abstract}

In this paper, we investigate weighted composition operators $W_{f,g}$ commuting with complex symmetric weighted composition operators $W_{\psi, \varphi}$ on the classical Fock space $\mathcal{F}^{2}(\mathbb{C})$. In particular, we investigate the symbols $f$ and $g$ give rise to $W_{f,g}$ commuting with a complex symmetric weighted composition operator on $\mathcal{F}^{2}$. We further characterize when the commuting weighted composition operators are self-adjoint and normal.

 \end{abstract}

\maketitle

\section{Introduction}\label{sec:intr}
Let $\mathcal{F}^{2}(\mathbb{C})$ denote the classical Fock space of entire functions on the complex plane $\mathbb{C}$. This space consists of all entire functions $f$ satisfying
\[
\|f\|^{2}
=
\frac{1}{\pi}
\int_{\mathbb{C}}
|f(z)|^{2}e^{-|z|^{2}}\,dA(z)
<
\infty,
\]
where $dA$ denotes the planar Lebesgue measure. The Fock space is a reproducing kernel Hilbert space with reproducing kernel
\[
K_{w}(z)=e^{\overline{w}z},
\qquad z,w\in\mathbb{C}.
\]
Because of its close connection with quantum mechanics, harmonic analysis, and operator theory, the Fock space has become an important framework for studying linear operators induced by analytic mappings. To know more about Fock spoaces, we refer \cite{KZ} and references therein.

Among the most widely studied operators on spaces of analytic functions are composition operators and weighted composition operators. Given entire functions $\psi$ and $\varphi$, the weighted composition operator $W_{\psi,\varphi}$ on $\mathcal{F}^{2}(\mathbb{C})$ is defined by
\begin{equation*}
W_{\psi,\varphi}h
=
\psi\cdot(h\circ\varphi),
\qquad h\in\mathcal{F}^{2}(\mathbb{C}).
\end{equation*}
Special cases include the multiplication operator $M_{\psi}$ when $\varphi(z)=z$, and the composition operator $C_{\varphi}$ when $\psi\equiv1$.
The action of adjoint of weighted composition operator on kernel of fock space is defined by
\begin{equation}\label{1}
    W^{*}_{\psi,\varphi}K_{w}=\overline{\psi} K_{\varphi(w)},\qquad w\in \mathbb{C}
\end{equation}

Weighted composition operators on the Fock space have been studied extensively in recent years. Their boundedness, compactness, invertibility, spectra, self-adjointness, normality, and unitary properties have been investigated by many authors; see for example \cite{GH,PVL,PVL2,TL,TL1,NST,LZ1,LZ2,LZ3}.A significant contribution was made by Trieu Le \cite{TL1}, who established simple and complete characterizations of bounded and compact weighted composition operators on the classical Fock space.


The systematic study of complex symmetric operators was initiated by Garcia and Putinar\cite{GP1}, who showed that many important classes of operators possess complex symmetric structures. Since then, complex symmetric weighted composition operators have been studied on Hardy spaces, Bergman spaces, and Fock spaces, revealing strong geometric restrictions on the inducing symbols, for references see \cite{GP1}-\cite{NST} and references therein.

A central topic in this paper is the study of commutant of complex symmetric operators.
Determining the commutant of weighted composition operators is generally a difficult problem because commuting relations often reduce to complicated functional equations involving the inducing symbols. Early contributions in this direction were made by Cowen\cite{CC1,CC2}, who studied commutants of analytic Toeplitz and composition operators. Later, Cload  \cite{WC} investigated commutants of composition operators in detail. More recently, E. Ko \cite{EK} characterized weighted composition operators commuting with self-adjoint weighted composition operators on the Hardy space $H^{2}$, while Bhuia \cite{SRB} studied operators commuting with complex symmetric weighted composition operators on Hardy spaces.

Motivated by these developments and the papers \cite{SRB} and \cite{EK}, in this paper we investigate the commutants of complex symmetric weighted composition operators on the Fock space $\mathcal{F}^{2}(\mathbb{C})$. We derive necessary and sufficient conditions for a weighted composition operator to commute with a given complex symmetric weighted composition operator. We further characterize the structure of operators belonging to the commutants and analyze special cases including normal and self-adjoint weighted composition operators. Our approach relies heavily on reproducing kernel techniques, functional equations, and explicit formulas for adjoint operators on the Fock space.

The organization of the paper is as follows. In Section 2, we present preliminaries concerning the Fock space, reproducing kernels, and weighted composition operators. In section 3, we establish the main commutant characterizations for complex symmetric weighted composition operators. In Section 4, discusses Self adjointness and Normality of commuting weighted composition operators and an example illustrating the obtained results.

\section{Preliminaries}
In this section, we collect several definitions and known results that will be used throughout the paper. Then, we recall some important facts concerning complex symmetric operators, conjugations, and affine inducing symbols on the Fock space.
We recall the following standard definitions:
\begin{Definition}
    A conjugation $C: H \longrightarrow H$ is a mapping on a complex Hilbert space $H$ which is
\begin{enumerate}
 \item[{(i)}]  Conjugate linear: $C(b x+\beta y)=\overline{b}C x+\overline{\beta}Cy$ , for all $x,y$ in $H$ and $b, \beta\in \mathbb{C}.$
 \item[{(ii)}] Involutive: $C^{2}=I \text{ and }$
 \item[{(iii)}] Isometric: $\|Cx\|=\|x\|$, for all $x\in H$.
 \end{enumerate}
\end{Definition}
Throughout the paper we will consider the Conjugation $J_{a,b,c}$ defined by
 \begin{equation}
   J_{a,b,c}h(z)=  ce^{bz}\overline{h(\overline{az+b})}, \text{ where } |a|=1, \overline{a}b+\overline{b}=0, |c|^{2}e^{|b|^{2}=1}
 \end{equation}
 In the special case $a=1$, $b=0$, and $c=1$, the conjugation $J_{a,b,c}$ reduces to the standard conjugation
 $J_{1,0,1}h(z)=\overline{h(\overline{z})}$
 considered in\cite{GH}.
 
\begin{Definition} Let $\mathcal{T}$ be a bounded linear operator, $\mathcal{T}^{*}$ be an adjoint of $\mathcal{T}$ and $C$ be a conjugation on Hilbert space $H$. Then 
\begin{enumerate}
    \item[{(i)}]$\mathcal{T}$ is $C$-complex symmetric on $H$ if $\mathcal{T}=C\mathcal{T}^{*}C$,
    \item[{(i)}]$\mathcal{T}$ is self-adjoint if $\mathcal{T}=\mathcal{T}^{*}$ and
    \item[{(i)}]$\mathcal{T}$ is normal if $\mathcal{T} \mathcal{T}^{*}=\mathcal{T}^{*}\mathcal{T}$ or $\|\mathcal{T}\|=\|\mathcal{T}^{*}\|$.  
\end{enumerate}  
\end{Definition}
\begin{Definition}
     For a bounded operator $T$, the commutant of $T$ is defined by
$$
\{T\}'=\{A\in\mathcal{B}(\mathcal{H}):AT=TA\}.
$$
\end{Definition}

We shall frequently use the following characterization of complex symmetric weighted composition operators on the Fock space.
\begin{Theorem}\cite{PVL}\label{t1}
  Let $W_{\psi, \varphi}$ be bounded weighted composition operator on $\mathcal{F}^{2}$ induced by two entire function $\psi$ and $\varphi$ with $\psi\not\equiv 0 $. Then  $W_{\psi, \varphi}$ is $J_{a,b,c}$-symmetric  if and only if
  $$\varphi(z)= Az+B,\quad \psi(z)=Ce^{Dz}$$ 
  with $C\neq 0$, $D=aB-bA+b$ where $A$, $B$ are complex constants satisfying either $|A|< 1$
  or $|A|=1, D+A\overline{B}=0$.
\end{Theorem}
To know more about the boundedness, compactness, self-adjointness and normality of weighted composition operators on Fock space, we refer \cite{TL1}and references therein.
The following eigenstructure result will also be needed.
\begin{Theorem}\cite{PVL}\label{t2}Let $W_{\psi, \varphi}$ be a $J_{a,b,c}$-symmetric bounded weighted composition operator on $\mathcal{F}^{2}$ induced by two entire function $\psi$ and $\varphi$ with $\psi\not\equiv 0 $. If $\varphi$ has a fixed point $d(=\frac{B}{1-A})$, then 
$$\Omega_{k}(z)= (z-d)^{k}K_{\overline{\alpha}}(z),\;\;\; k\in \mathbb{N}\cup \{0\}, \alpha=\frac{D}{1-A}.  $$
are eigenvectors of $W_{\psi, \varphi}$ corresponding to the eigenvalues $A^{k}\psi(d)$.
\end{Theorem}

\section{Commutants of complex symmetric weighted composition operators}\label{sec:prelimi}
In this section, we investigate the commutant of a complex symmetric weighted composition operator. We derive necessary and sufficient conditions for two weighted composition operators to commute. The commuting relation leads to functional equations involving the symbols, and these equations are analyzed using reproducing kernel techniques and analytic function theory. To establish the main theorem, we first require the following auxiliary lemmas.
\begin{lemma}\label{1l}
    Let $W_{f,g}$ and  $W_{\psi, \varphi}$ be two weighted composition operators on $\mathcal{F}^{2}$, where $f, g, \psi, \varphi$ are entire functions. Then $W_{f,g}\in \{W_{\psi, \varphi}\}^{'}$ if and only if $f.\psi \circ g =\psi .f\circ \phi$ and $\varphi \circ g = g \circ \varphi$.
\end{lemma}
\begin{proof}
    We compute
    $$W_{f,g}W_{\psi, \varphi}= W_{f.\psi \circ g , \varphi \circ g} ~\text{ and }~ W_{\psi, \varphi}W_{f,g}= W_{\psi. f\circ \varphi , g\circ \varphi}. $$ Since we have
    $$W_{f,g}\in \{W_{\psi, \varphi}\}^{'} \text{ if and only if } W_{f,g}W_{\psi, \varphi}= W_{\psi, \varphi}W_{f,g}.$$
    Hence $W_{f,g}\in \{W_{\psi, \varphi}\}^{'}$
    if and only if $f.\psi \circ g =\psi .f\circ \phi$ and $\varphi \circ g = g \circ \varphi$.
\end{proof}
\begin{lemma}\label{l2} 
Let $f$ and $g$ be two entire function with $f\neq 0$ and $W_{\psi, \varphi}$ be $J_{a,b,c}$-complex symmetric weighted  composition operators on $\mathcal{F}^{2}$ with $\varphi$ has a fixed point 
$0\neq d(=\frac{B}{1-A})$ in $\mathbb{C}$. If $W_{f,g} \in \{W_{\psi, \varphi}\}^{'}$, then  $\varphi$ and $g$ have a same fixed point $d$.
\end{lemma}
\begin{proof}
   Since by lemma\ref{1l}, we have 
   $$W_{f,g} \in \{W_{\psi, \varphi}\}^{'} ~\text{ if and only if }~ f.\psi \circ g =\psi .f\circ \phi ~\text{ and }~ \varphi \circ g = g \circ \varphi $$

   and by Theorem \ref{t1} $W_{\psi, \varphi}$ is $J_{a,b,c}$-complex symmetric weighted  composition operators with $\varphi$ has a fixed point 
$0\neq d(=\frac{B}{1-A})$ in $\mathbb{C}$. Then we have 
$$\varphi(g(d)) = g (\varphi(d))= g(d)$$
which implies $g(d)=d$. This proves that $\varphi$ and $g$ have a same fixed point $d$.
\end{proof}

\begin{lemma}\label{l1}Let $\eta= \alpha(1-\beta)$, $\alpha=\frac{D}{1-A}$, $d=\frac{B}{1-A}$ and $\beta \in \mathbb{C}$, Then the following  identities hold:
\begin{enumerate}
\item $D+\alpha A = \alpha$ and $\alpha B= dD$
\item $\eta+D\beta= D+\eta A$ 
\item $D(1-\beta)d= \eta B$ 
\item $ A(1-\beta)d+B= (1-\beta)d+B\beta$
\end{enumerate}
\end{lemma}
\begin{proof}
    
       \begin{enumerate}
           \item $D+\alpha A= D+\frac{D}{1-A}A=\frac{D}{1-A}=\alpha$ and $\alpha B= \frac{D}{1-A} B= dD$\\
       \item We compute
       \begin{align*}
            \eta+D\beta &= \alpha(1-\beta)+D\beta\\
            &=\frac{D}{1-A}(1-\beta)+D\beta\\
            &= \frac{D-DA\beta}{1-A}
        \end{align*}  
        and
            \begin{align*}
            D+\eta A &= D+\alpha (1-\beta)A\\
            &= D+\frac{D}{1-A}(1-\beta)A\\
            &=\frac{D-DA\beta}{1-A}
        \end{align*}
        
        Therefore $\eta+D\beta= D+\eta A$ holds.\\
        \item  We compute
        \begin{align*}
        D(1-\beta)d &= \frac{D(1-\beta)B}{1-A}\\
        &= \alpha(1-\beta)B\\
        &= \eta B 
        \end{align*}
    Therefore $D(1-\beta)d= \eta B$ holds.\\
    
    \item  We compute
        \begin{align*}
        A(1-\beta)d+B &= \frac{A(1-\beta)B}{1-A}+B\\
        &= \frac{B(1-AB)}{1-A}
    \end{align*}
    and 
       \begin{align*}
        (1-\beta)d+B\beta &= (1-\beta)\frac{B}{1-A}+ B\beta\\
        &= \frac{B(1-AB)}{1-A}
    \end{align*} 
    
Therefore $ A(1-\beta)d+B= (1-\beta)d+B\beta$ holds.
\end{enumerate}
\end{proof}

\begin{Theorem}\label{t3}
Let $f$ and $g$ be two entire function with $f\neq 0$ and $W_{\psi, \varphi}$ be $J_{a,b,c}$-complex symmetric weighted  composition operators on $\mathcal{F}^{2}$ with $\varphi$ has a fixed point 
$d(=\frac{B}{1-A})$ in $\mathbb{C}$. Then $W_{f,g} \in \{W_{\psi, \varphi}\}^{'}$ if and only if symbols functions $f$ and $g$ are of the following forms:
\begin{equation}\label{e1}
  g(z)= \beta z+ (1-\beta) d \text{~~ and ~~} f(z)= f(d) e^{\eta(z-d)}  
\end{equation}

where  $\beta \in \mathbb{C}$ and $ \eta = \alpha (\beta -1)$ .
\end{Theorem}
\begin{proof}
First, by using identity $(1)$ of lemma\ref{l1}, we compute
\begin{align*}
    W_{\psi, \varphi}K_{\overline{\alpha}}(z)&= C.e^{Dz}e^{\alpha \varphi(z)}\\
    &= C. e^{(D+\alpha A)z+\alpha B}\\
    &= \psi(d)K_{\overline{\alpha}}(z)
\end{align*}

  Since $W_{f,g} \in \{W_{\psi, \varphi}\}^{'}$ and $W_{\psi, \varphi}$ be $J_{a,b,c}$-complex symmetric , then by theorem\ref{t2} we have

  \begin{align*}
  W_{\psi, \varphi}W_{f,g}K_{\overline{\alpha}}&= W_{f,g}W_{\psi, \varphi}K_{\overline{\alpha}}\\
  &= \psi(d)W_{f,g}K_{\overline{\alpha}}
  \end{align*}
If $W_{f,g}K_{\overline{\alpha}}=0$, then $K_{\overline{\alpha}}$ is an eigen vector of $W_{f,g}$ corresponding to eigen value 0.
If $W_{f,g}K_{\overline{\alpha}}\neq 0$, then $W_{f,g}K_{\overline{\alpha}}$ is an eigen vector of $W_{f,g}$ corresponding to eigen value 
$\psi(d)$. Then by theorem\ref{t1} we have $W_{f,g}K_{\overline{\alpha}}= \gamma_{1}K_{\overline{\alpha}}$ for some non zero $\gamma_{1}\in \mathbb{C}$. Again by theorem\ref{t1}, the function $\Omega_{1}= (z-d)K_{\overline{\alpha}}$ is an eigen vector of $W_{\psi, \varphi}$ corresponding to eigen value $A\psi(d)$.
Since $W_{f,g} \in \{W_{\psi, \varphi}\}^{'}$, we have 
\begin{align*}W_{\psi, \varphi}W_{f,g}\Omega_{1}&=W_{f,g}W_{\psi, \varphi}\Omega_{1}\\
&=A\psi(d) W_{f,g}\Omega_{1} 
\end{align*}

This show that $W_{f,g}\Omega_{1}$  is an eigenvector corresponding to the eigenvalue $A\psi(d)$. Therefore $W_{f,g}\Omega_{1}=\gamma_{2}\Omega_{1}$, for some $\gamma_{2} \in \mathbb{C}$. Thus we have

\begin{align*}
   W_{f,g}\Omega_{1}&=\gamma_{2}\Omega_{1}\\
    f(z) (g(z)-d)K_{\overline{\alpha}}(g(z))&= \gamma_{2} (z-d)K_{\overline{\alpha}}(z)\\
    (W_{f,g}K _{\overline{\alpha}}(z))(g(z)-d))&= \gamma_{2}(z-d)K_{\overline{\alpha}}(z)\\
    \gamma_{1}K_{\overline{\alpha}}(z)(g(z)-d)&=\gamma_{2} (z-d)K_{\overline{\alpha}}(z)\\
    g(z)&=\beta(z-d)+d\\
    g(z)&=\beta z+ (1-\beta) d
\end{align*}
where $\beta=\frac{\gamma_{2}}{\gamma_{1}} \in \mathbb{C}.$

Now, let $z=d$ in  $W_{f,g}K_{\overline{\alpha}}(z)= \gamma_{1}K_{\overline{\alpha}}(z)$, by lemma\ref{l2}, we get that $\gamma_{1}=f(d)$. Therefore, we have
\begin{align*}
    W_{f,g}K_{\overline{\alpha}}(z)&= \gamma_{1}K_{\overline{\alpha}}(z)\\
   f(z)e^{\alpha g(z)}&=f(d)e^{\alpha z} \\
   f(z)&= f(d)e^{\alpha z-\alpha g(z)}\\
   f(z)&= f(d) e^{\alpha z-\alpha(\beta z+ (1-\beta) d)}\\
   f(z)&= f(d) e^{\eta(z-d)}
\end{align*}
where $\eta = \alpha(1-\beta)$.\\
Conversely, suppose that $f$ and $g$ are of the form in \eqref{e1}. To complete the proof we have to show that 
$$f.\psi \circ g = \psi. f\circ \varphi \text{ and }\varphi \circ g= g\circ \varphi.$$
Thus we have 
\begin{align*}
 f(z)\psi (g(z)) &= f(d) e^{\eta(z-d)} C. e^{Dg(z)}\\
 &= C.f(d)e^{\eta z-\eta d+D(\beta z +(1-\beta)d)}\\
 &= C.f(d)e^{(\eta+D\beta)z+(D(1-\beta)d-\eta d)}
\end{align*}
and 
\begin{align*}
    \psi(z) f( \varphi (z))&= C. e^{Dz} f(d)e^{\eta [(Az+B)-d]}\\
    &= C. e^{Dz} f(d) e^{(D+\eta A)z+(\eta B-\eta d)}
\end{align*}
Clearly, by using identities $(1)$ and $(2)$ of lemma\ref{l1}, we conclude that $f.\psi \circ g = \psi. f\circ \varphi $. Now we compute
\begin{align*}
    \varphi( g(z))&= Ag(z)+B\\
    &= A(\beta z+ (1-\beta) d)+B
\end{align*}
and 
\begin{align*}
    g(\varphi(z))&= \beta \varphi(z)+ (1-\beta) d\\
    &= \beta (Az+B)+(1-\beta) d\\
\end{align*}
Again by using identity $(3)$ of lemma\ref{l1}, we conclude that $\varphi \circ g= g\circ \varphi.$ This completes the proof of the theorem.
\end{proof}

\section{Self adjointness and Normality of commuting weighted composition operators}
This section focuses on commuting weighted composition operators that are self-adjoint or normal. We characterize the symbols that generate such operators and establish rigidity results for affine composition symbols. Several consequences concerning spectral properties and fixed-point structures are also obtained.

\begin{lemma}\label{l3}Let $f$ and $g$ be two entire symbols of the form \eqref{e1} where $W_{f,g}$ be a weighted composition operator commuting with $J_{a,b,c}$-complex symmetric weighted  composition operators $W_{\psi, \varphi}$ on $\mathcal{F}^{2}$ with $\varphi$ has a fixed point 
$d(=\frac{B}{1-A})$ in $\mathbb{C}$. Then 
\begin{itemize}
\item[{(a)}] $W_{f,g}K_{w}(z)=f(d) e^{ \beta \overline{w}z+\eta z+(1-\beta)d \overline{w}-\eta d}$ and $W^{*}_{f,g}K_{w}(z)=\overline{f(d)}e^{\overline{\beta}\overline{w}z+\overline{\eta}\overline{w}+(1-\overline{\beta}) \overline{d}z- \overline{\eta}\overline{d}}$
    \item[{(b)}] $W_{f,g}W^{*}_{f,g}K_{w}(z)= |f(d)|^{2}e^{-\overline{\eta}\overline{d}-\eta d+|1-\beta|^{2}|d|^{2}+|\beta|^{2}\overline{w}z+[\overline{\eta}+\overline{\beta}(1-\beta)d]\overline{w}+[ \eta +\beta(1-\overline{\beta})\overline{d}]z}$\\
    \item[{(c)}] $W^{*}_{f,g}W_{f,g}K_{w}(z)=|f(d)|^{2}e^{|\eta|^{2}-\eta d-\overline{\eta}\overline{d}+|\beta|^{2}\overline{w}z+[(1-\beta)d+\overline{\eta}\beta]\overline{w}+[\overline{\beta} \eta +(1-\overline{\beta})\overline{d}]z}$
\end{itemize} 
\end{lemma}
\begin{proof} Since $f$ and $g$ are of the form
\begin{equation}\label{e1}
  g(z)= \beta z+ (1-\beta) d \text{~~ and ~~} f(z)= f(d) e^{\eta(z-d)}  
\end{equation}
where  $\beta \in \mathbb{C}$ and $ \eta = \alpha (\beta -1)$ We compute the following:
\begin{itemize}
\item[{(a)}]
\begin{equation*}
        W_{f,g}K_{w}(z)=f(z)e^{wg(z)}= f(d)e^{\eta(z-d)}e^{\overline{w}( \beta z+ (1-\beta) d)}= f(d) e^{ \beta \overline{w}z+\eta z+(1-\beta)d \overline{w}-\eta d}
    \end{equation*}
    and
\begin{equation*}
    W^{*}_{f,g}K_{w}(z)=\overline{f(w)}e^{\overline{g(w)}z}=\overline{f(d)}e^{\overline{\eta(w-d)}}e^{[\overline{\beta w+ (1-\beta) d}]z}= \overline{f(d)}e^{\overline{\beta}\overline{w}z+\overline{\eta}\overline{w}+(1-\overline{\beta}) \overline{d}z- \overline{\eta}\overline{d}}
\end{equation*}

    \item[{(b)}] \begin{align}
       W_{f,g}W^{*}_{f,g}K_{w}(z)&= \overline{f(w)}W_{f,g}K_{g(w)}(z)\notag\\
       &=\overline{f(w)}f(z)K_{g(w)}(g(z))\notag\\
       &= \overline{f(d)}e^{\overline{\eta(w-d)}}f(d)e^{\eta(z-d)}e^{\overline{[\beta w+(1-\beta)d]}[\beta z+(1-\beta)d]} \notag\\
       &= |f(d)|^{2}e^{|1-\beta|^{2}|d|^{2}-\overline{\eta}\overline{d}-\eta d}e^{|\beta|^{2}\overline{w}z} e^{[\overline{\eta}+\overline{\beta}(1-\beta)d]\overline{w}}e^{[\eta+\beta(\overline{1-\beta})\overline{d}]z}
    \end{align}
and

\item[{(c)}] \begin{align}
    W^{*}_{f,g}W_{f,g}K_{w}(z)&= f(d)e^{(1-\beta)d\overline{w}-\eta d}W^{*}_{f,g} K_{\overline{\beta}w+\overline{\eta}}(z)\notag\\
    &= f(d)e^{(1-\beta)d\overline{w}-\eta d} \overline{f(\overline{\beta}w+\overline{\eta})}K_{g(\overline{\beta}w+\overline{\eta})}(z)\notag\\
    &= f(d)e^{(1-\beta)d\overline{w}-\eta d}\overline{f(d)}e^{\overline{\eta(\overline{\beta}w+\overline{\eta}-d)}}e^{[g(\overline{\beta}w+\overline{\eta})]z}\notag\\
    &= |f(d)|^{2}e^{|\eta|^{2}-\overline{\eta}\overline{d}-\eta d}e^{|\beta|^{2}\overline{w}z} e^{[\overline{\eta}\beta+(1-\beta)d]\overline{w}}e^{[\eta\overline{\beta}+(\overline{1-\beta})\overline{d}]z}
\end{align}
  \end{itemize}  
\end{proof}

\begin{Theorem}\label{t4}
    Let $f$ and $g$ be two entire symbols of the form \eqref{e1} where $W_{f,g}$ be a weighted composition operator commuting with $J_{a,b,c}$-complex symmetric weighted  composition operators $W_{\psi, \varphi}$ on $\mathcal{F}^{2}$ with $\varphi$ has a fixed point 
$d(=\frac{B}{1-A})$ in $\mathbb{C}$. Then 
\begin{enumerate}
\item  $W_{f,g}$ is self-adjoint if and ony if $d= \alpha$ and $f(d),\beta,d,\alpha \in \mathbb{R}$. 
    \item $W_{f,g}$ is normal if and only if $\alpha=\overline{d}$.  
\end{enumerate}

\end{Theorem}

\begin{proof}
\begin{enumerate}
    \item Since $W_{f,g}$ is self-adjoint if and only if $W_{f,g}K_{w}(z)=W^{*}_{f,g}K_{w}(z)$. Therefore by lemma\ref{l3}, $W_{f,g}$ is self adjoint if and only if
$f(d)=\overline{f(d)}$ and 
\begin{equation}\label{e2}
\beta \overline{w}z+\eta z+(1-\beta)d \overline{w}-\eta d = \overline{\beta}\overline{w}z+\overline{\eta}\overline{w}+(1-\overline{\beta}) \overline{d}z- \overline{\eta}\overline{d}
\end{equation}
let $z=w=0$ in \eqref{e2}, we get $\eta d =\overline{\eta}\overline{d}$ implies $\eta=\overline{\eta}$ and $d= \overline{d}$. 
Now take $z=0$, we get 
    $(1-\beta)d =\overline{\eta}$ implies  $d=\alpha$.
Thus  $W_{f,g}$ is self-adjoint if and only if  $d= \alpha$ and $f(d),\beta,d,\alpha \in \mathbb{R}$.    

\item Since $W_{f,g}$ is normal if and only if $W_{f,g}W^{*}_{f,g}K_{w}(z)=W^{*}_{f,g}W_{f,g}K_{w}(z)$. Therefore by lemma\ref{l3} $W_{f,g}$ is normal if and only if 
\begin{align}\label{e3}
  &|1-\beta|^{2}|d|^{2}-\overline{\eta}\overline{d}-\eta d+|\beta|^{2}\overline{w}z+[\overline{\eta}+\overline{\beta}(1-\beta)d]\overline{w}+[ \eta +\beta(1-\overline{\beta})\overline{d}]z \notag \\
  &= |\eta|^{2}-\eta d-\overline{\eta}\overline{d}+|\beta|^{2}\overline{w}z+[(1-\beta)d+\overline{\eta}\beta]\overline{w}+[\overline{\beta} \eta +(1-\overline{\beta})\overline{d}]z
 \end{align}
let $z=w=0$ in \eqref{e3}, we get $|d|=|\alpha|$. 
 Now, take $z=0$ in \eqref{e3}, we get $\overline{\eta}+\overline{\beta}(1-\beta)d]= (1-\beta)d+\overline{\eta}\beta$, implies $d=\overline{\alpha}$.
Hence, $W_{f,g}$ is normal if and only if $\alpha=\overline{d}$.
\end{enumerate}
\end{proof}
\begin{ex}
   Consider $$\varphi(z)=\frac{z}{2}+1 ~\text{ and } ~\psi(z)= \frac{e^{\frac{iz}{2}}}{2}$$ as entire function, where $A=\frac{1}{2}, B=1, C= \frac{1}{2}, D= \frac{i}{2}, d= 2,$ and $ \alpha =i.$\\
   If we choose $a=0,b=i,$ and $c=\frac{1}{2}$, then it is easy to check $W_{\psi, \varphi}$ is $J_{a,b,c}$- symmetric on $F^{2}(\mathbb{C})$ because $$|A|<1~\text{ and }~ D=\frac{i}{2}= aB-bA+b.$$
   Now, consider the entire functions $$g(z)= iz+(1-i)2~ \text{ and }~ f(z)=2e^{(-1-i)(z-2)} $$ where $\beta=i$. One can easily check, 
   $$ W_{f,g}\in \{W_{\psi, \varphi}\}^{'}$$
   because $\varphi\circ g= g\circ \varphi$ and $f.\psi\circ g= \psi .f\circ \varphi$ holds but $ W_{f,g}$ is neither self-adjoint nor normal because $i=\alpha \neq d=2$.
\end{ex}

\section{Conclusion}
In this paper, we studied commutants of complex symmetric weighted composition operators on the classical Fock space. Necessary and sufficient conditions for commuting relations were obtained, and structural properties of the inducing symbols were established. We also characterized self-adjoint and normal commuting weighted composition operators. The results presented here contribute to the growing literature on operator theory in Fock spaces and suggest several directions for future research, including extensions to several-variable Fock spaces and other reproducing kernel Hilbert spaces.
Several natural questions remain open. In particular:
\begin{itemize}
\item What happens when the inducing map $\varphi$ is nonlinear?
\item Can one characterize commutants on the Fock space of several variables?
\item What are the corresponding results for unitary or hyponormal weighted composition operators?

\end{itemize}

\section{conflict of interest}
There is no conflict of interest.

\end{document}